\documentclass{amsproc}
\usepackage{eurosym}
\usepackage{amssymb}
\usepackage{amsfonts}

\theoremstyle{plain}

\newtheorem{theorem}{Theorem}
\numberwithin{equation}{section}

\begin{document}
\title{Indicated coloring of matroids}
\author{Micha\l\ Laso\'{n}}
\address{Institute of Mathematics of the Polish Academy of Sciences, \'{S}niadeckich 8, 00-956 Warszawa, Poland}
\address{Theoretical Computer Science Department, Faculty of Mathematics and
Computer Science, Jagiellonian University, \L ojasiewicza 6, 30-348 Krak\'{o}w, Poland}
\email{michalason@gmail.com}
\thanks{This publication is supported by the Polish National Science Centre grant no. 2011/03/N/ST1/02918.}

\begin{abstract}
A coloring of a matroid is \emph{proper} if elements of the same color form
an independent set. For a loopless matroid $M$, its chromatic number $\chi (M)$
is the minimum number of colors that suffices to color properly the ground set $E$
of $M$. In this note we study a game-theoretic variant of this parameter proposed by Grytczuk. 
Suppose that in each round of the game Alice indicates an uncolored yet element $e$ of $E$, then Bob colors it using a color from a fixed set of colors $C$. The rule Bob has to obey is that it is a proper coloring. The game ends if the whole matroid has been colored or if Bob can not color $e$ using any color of $C$. Alice wins in the first case, while Bob in the second. The minimum size of the set of colors $C$ for which Alice has a winning strategy is called the \emph{indicated chromatic number} of $M$, denoted by $\chi_{i}(M)$. We prove that $\chi_i(M)=\chi(M)$.
\end{abstract}

\maketitle

%%%%%%%%%%%%%%%%%%%%%%%%%%%%%%%%%%%%%%%%%%%%%%%%%%%%%%%%%%%%%%%%%%%%%%%%%%%%%%%%%%%%%%%%%%%%%%%%%%%%%%%%%%%%%%%%%%%%%%%%%%%%%%%%%%%%%%%%%%%%%%%%%%%%%%%%%%%%%%%%%%%%%%%%%
\section{Introduction}
%%%%%%%%%%%%%%%%%%%%%%%%%%%%%%%%%%%%%%%%%%%%%%%%%%%%%%%%%%%%%%%%%%%%%%%%%%%%%%%%%%%%%%%%%%%%%%%%%%%%%%%%%%%%%%%%%%%%%%%%%%%%%%%%%%%%%%%%%%%%%%%%%%%%%%%%%%%%%%%%%%%%%%%%%%

Let $M$ be a loopless matroid on a ground set $E$ (the reader is referred to \cite{Ox92} for background of matroid theory). In analogy to the graph case we say that a coloring of the set $E$ is \emph{proper} if elements of the same color form an independent set of $M$. The \emph{chromatic number} of $M$, denoted by $\chi (M)$, is the minimum number of colors that suffice to color properly the set $E$. In case of a graphic matroid $M=M(G)$, the number $\chi (M)$ is a well studied
parameter known as the \emph{arboricity} of the underlying graph $G$.

The study of game-theoretic variants of chromatic number was initiated for graphs independently by Brams, cf. \cite{Ga81}, and Bodlaender \cite{Bo91}. They defined game chromatic number of a graph, which was intensively studied (see \cite{AlAl89,BaGr08,DiZh99,Ki00,Zh99}). A natural question concerning all game-theoretic variants of chromatic number is whether it is bounded from above by a function of chromatic number, and if yes then what is the best possible bound. The game chromatic number of a graph is not bounded, as it can be arbitrary large for bipartite graphs. A matroidal version---the game chromatic number $\chi_{g}(M)$ of a matroid was studied in \cite{La12}, where the author shows that $\chi_{g}(M)\leq 2\chi (M)$ for every matroid $M$. This gives a nearly tight bound, since for every $k\geq 3$ there are matroids with $\chi (M)=k$ and $\chi_{g}(M)\geq 2k-1$. 

Another game-theoretic variant of list chromatic number was introduced by Schauz \cite{Sc09} (see also \cite{Zh09}), it is called on-line list chromatic number. For graphs it is bounded by an exponential function of chromatic number, and there are known examples for which both parameters differ. For matroids they are always equal \cite{LaLu12}, and also equal to the chromatic number \cite{Se98}.

The newest variant of the graph coloring game was proposed by Grytczuk. Let $G$ be a graph, and let $C$ be a fixed set of colors. In each round of the game Alice indicates an uncolored yet vertex, then Bob colors it using a color from $C$. The only rule Bob has to obey is that it is a proper coloring. The goal of Alice is to achieve a proper coloring of the whole graph, while Bob is trying to prevent it (arrive at a partial coloring that can not be extended). The minimum size of the set of colors $C$ for which Alice has a winning strategy is called the \emph{indicated chromatic number} of a graph $G$, denoted by $\chi_{i}(G)$. Clearly $\chi_i(G)\geq\chi(G)$.
Grzesik \cite{Gr12} proved that if $\chi(G)=2$, then $\chi_i(G)=2$ and gave an example of a graph $G$ with $\chi(G)=3$ and $\chi_i(G)=4$. He also shows an upper bound $\chi_i(G)\leq 4\chi(G)$ for random graphs, and conjectures that indicated chromatic number of a graph is bounded by a function of chromatic number.

In this paper we study a matroidal version of the indicated chromatic number. Our main result reads as follows.

\begin{theorem}\label{main}
Every loopless matroid $M$ satisfies $\chi_i(M)=\chi(M)$.
\end{theorem}

The proof is by induction used to a suitable generalization of the game. We end the paper with a fancy modification of indicated chromatic number also made by Grytczuk.

%%%%%%%%%%%%%%%%%%%%%%%%%%%%%%%%%%%%%%%%%%%%%%%%%%%%%%%%%%%%%%%%%%%%%%%%%%%%%%%%%%%%%%%%%%%%%%%%%%%%%%%%%%%%%%%%%%%%%%%%%%%%%%%%%%%%%%%%%%%%%%%%%%%%%%%%%%%%%%%%%%%%%%%%%
\section{Indicated chromatic number}
%%%%%%%%%%%%%%%%%%%%%%%%%%%%%%%%%%%%%%%%%%%%%%%%%%%%%%%%%%%%%%%%%%%%%%%%%%%%%%%%%%%%%%%%%%%%%%%%%%%%%%%%%%%%%%%%%%%%%%%%%%%%%%%%%%%%%%%%%%%%%%%%%%%%%%%%%%%%%%%%%%%%%%%%%%

The main tool we use is the matroid union theorem (for a proof see \cite{Ox92}).

\begin{theorem}\emph{(Matroid Union Theorem)}\label{union}
Let $M_1,\dots,M_k$ be matroids on the same ground set $E$, with rank functions $r_1,\dots,r_k$ respectively. The following conditions are equivalent:
\begin{enumerate}
\item there exist sets $V_i$ with $V_1\cup\dots\cup V_k=E$, such that for each $i$ the set $V_i$ is independent in $M_i$,
\item for each $A\subset E$ holds $r_1(A)+\dots+r_k(A)\geq\lvert A\rvert$.
\end{enumerate}
\end{theorem}

For seek of completeness we repeat the definition of the game in a matroid setting. Let $M$ be a matroid on a ground set $E$. In each round of the game Alice indicates an uncolored yet element of $E$, then Bob colors it using a color from a fixed set of colors $C$. The rule Bob has to obey is that it is a proper coloring. The game ends if the whole matroid has been colored or if Bob can not color $e$ using any color of $C$. Alice wins in the first case, while Bob in the second. The \emph{indicated chromatic number} of a matroid $M$, denoted by $\chi_i(M)$, is the minimum size of the set of colors $C$ for which Alice has a winning strategy. Clearly $\chi_i(M)\geq\chi(M)$.

Theorem \ref{main} is a corollary of a slightly more general Theorem \ref{general} which concerns a game explained below. 

Let $M_1,\dots,M_k$ be a collection of matroids on the same ground set $E$. We consider a modification of the game in which the set of colors is $\{1,\dots,k\}$. The other difference concerns the rule Bob has to obey. Namely elements of color $i$ must form an independent set in a matroid $M_i$. As usual, Alice wins if the whole set $E$ has been colored, while Bob wins if a partial coloring can not be properly extended.

\begin{theorem}\label{general}
Let $M_1,\dots,M_k$ be matroids on the same ground set $E$. Suppose there are sets $V_1,\dots,V_k$, such that $V_i$ is independent in $M_i$, and $V_1\cup\dots\cup V_k=E$. Then Alice has a winning strategy in the generalized indicated coloring game.
\end{theorem}

\begin{proof}
The proof goes by induction on the number of elements of $E$. For a set $E$ consisting of only one element assertion clearly holds. Thus we assume that $\lvert E\rvert>1$ and that the assertion is true for all smaller sets. 

Let us denote by $M\vert_{A}$ the restriction of a matroid $M$ to a set $A$, and by $M/A$ the contraction of a set $A$ in a matroid $M$. Denote also the rank functions of matroids $M_1,\dots,M_k$ respectively by $r_1,\dots,r_k$. From Theorem \ref{union} we know that for each $\emptyset\neq A\subset E$ holds an inequality $$r_1(A)+\dots+r_k(A)\geq\lvert A\rvert.$$ There are two cases:

Case 1: there is a proper subset $\emptyset\neq A\subsetneq E$ with equality $$r_1(A)+\dots+r_k(A)=\lvert A\rvert.$$ Then by subtracting equality for $A$ from the inequality for $A\cup B$ we get that for every subset $B\subset E\setminus A$ holds 
$$r_1(A\cup B)-r_1(A)+\dots+r_k(A\cup B)-r_k(A)\geq\lvert A\cup B\rvert-\lvert A\rvert=\lvert B\rvert.$$
The left side of this inequality is the sum of ranks of the set $B$ in matroids $M_1/A,\dots,M_k/A$. Thus by Theorem \ref{union} the collection of matroids $M_1/A,\dots,M_k/A$ on $E\setminus A$ satisfies assumptions of the theorem. The collection $M_1\vert_{A},\dots,M_k\vert_{A}$ of matroids on $A$ clearly also does. By inductive assumption Alice has a winning strategy in the game with matroids $M_1\vert_{A},\dots,M_k\vert_{A}$ on the set $A$, so she plays with this strategy. Let us denote elements colored by Bob with $i$ after the game is finished by $U_i$. Now the play moves to the set $E\setminus A$, so now the original game is on matroids $M_i/U_i$. But since the collection of matroids $M_i/A$ satisfies assumptions of the theorem, collection of matroids $M_i/U_i$ also does and Alice has a winning strategy by inductive assumption. As a result she wins the whole game on $E$. 

Case 2: for all $\emptyset\neq A\subsetneq E$ holds $$r_1(A)+\dots+r_k(A)>\lvert A\rvert.$$ Then in the first round of the game Alice indicates an arbitrary element $e\in E$. Obviously $e\in V_l$ for some $l$, thus Bob has an admissible move -- he can color it with $l$. Suppose Bob colors $e$ with $j$. Now the original game is on the matroids $M_i\vert_{E\setminus \{e\}}$ for $i\neq j$ and $M_j/\{e\}$ on the set $E\setminus \{e\}$. For them the second condition of Theorem \ref{union} holds since $r_j(A)$ can possibly be lower only by one. Hence Alice has a winning strategy by inductive assumption. 
\end{proof} 

%%%%%%%%%%%%%%%%%%%%%%%%%%%%%%%%%%%%%%%%%%%%%%%%%%%%%%%%%%%%%%%%%%%%%%%%%%%%%%%%%%%%%%%%%%%%%%%%%%%%%%%%%%%%%%%%%%%%%%%%%%%%%%%%%%%%%%%%%%%%%%%%%%%%%%%%%%%%%%%%%%%%%%%%%
\section{Modified indicated chromatic number}
%%%%%%%%%%%%%%%%%%%%%%%%%%%%%%%%%%%%%%%%%%%%%%%%%%%%%%%%%%%%%%%%%%%%%%%%%%%%%%%%%%%%%%%%%%%%%%%%%%%%%%%%%%%%%%%%%%%%%%%%%%%%%%%%%%%%%%%%%%%%%%%%%%%%%%%%%%%%%%%%%%%%%%%%%%

We consider a variant of the game from the previous section. Alice and Bob make alternative moves. If it is her turn, then as before, Alice indicates an uncolored yet element of $E$, and Bob colors it using a color from a fixed set $C$. While if it is Bob's turn their roles are swapped, now he indicates an uncolored yet element of $E$, and Alice colors it using a color from $C$. The other conditions of the game stay unchanged -- elements of the same color must form an independent set and Alice wins if the whole matroid has been colored. The minimum size of the set of colors $C$ for which Alice has a winning strategy we call a \emph{modified indicated chromatic number} of $M$, and denote by $\chi_{i}^{mod}(M)$. 

\begin{theorem}
Every loopless matroid $M$ satisfies $\chi_{i}^{mod}(M)=\chi(M)$.
\end{theorem}

As before in order to prove the above theorem we introduce a more general game. The theorem is a simple corollary of Theorem \ref{general2}. 

Let $M_1,\dots,M_k$ be matroids on the same ground set $E$. We consider a game in which in each turn Bob decides which of the following two kinds of moves is played:
\begin{enumerate}
\item Alice indicates an uncolored yet element of $E$ and Bob colors it using a color from $\{1,\dots,k\}$,
\item Bob indicates an uncolored yet element of $E$ and Alice colors it using a color from $\{1,\dots,k\}$. 
\end{enumerate}
Both players have to obey the rule that for each $i$ elements colored with $i$ form an independent set in the matroid $M_i$. If in some turn Alice or Bob does not have an admissible move the game ends. Alice wins if the whole matroid has been colored, while Bob wins if a partial coloring can not be properly extended.

\begin{theorem}\label{general2}
Let $M_1,\dots,M_k$ be matroids on the same ground set $E$. Suppose there are sets $V_1,\dots,V_k$, such that $V_i$ is independent in $M_i$, and $V_1\cup\dots\cup V_k=E$. Then Alice has a winning strategy in the generalized modified indicated coloring game.
\end{theorem}

\begin{proof}
The proof goes by induction on the number of elements of $E$. For a set $E$ consisting of only one element assertion clearly holds. Thus we assume that $\lvert E\rvert>1$ and that the assertion is true for all smaller sets. 

If Bob decided that in the first turn second kind of move is played, then he points say $e\in E$. From the assumption $e\in V_j$ for some $j$. Alice strategy is to color $e$ with $j$. Now the remaining part is played on matroids $M_i\vert_{E\setminus \{e\}}$ for $i\neq j$ and $M_j/\{e\}$, so assumptions of the theorem are clearly satisfied and from inductive assumption Alice has a winning strategy.

If Bob decided that in the first turn first kind of move is played, then consider generalized indicated coloring game on matroids $M_1,\dots,M_k$. By Theorem \ref{general} Alice has a winning strategy in this game. Suppose she indicates an element $e\in V_l$, and suppose Bob colors $e$ with $j$. Now the generalized indicated coloring game is on the matroids $M'_i=M_i\vert_{E\setminus \{e\}}$ for $i\neq j$ and $M'_j=M_j/\{e\}$. Moreover, Alice still has a winning strategy, thus there exists a proper coloring in which elements of color $i$ form an independent set in $M'_i$. For each $i$ let $U_i$ be the set of elements colored with $i$, we have $U_1\cup\dots\cup U_k=E\setminus\{e\}$. Since the remaining part of generalized modified indicated coloring game is played also on matroids $M'_i$, this gives that the assumptions of the theorem are satisfied. Thus Alice has a winning strategy by inductive assumption. 
\end{proof}

\end{document}